\documentclass[12pt]{amsart}

\usepackage{ucs}

\usepackage{amssymb}
\usepackage{amsthm}
\usepackage{amsmath}
\usepackage{latexsym}
\usepackage[cp1251]{inputenc}
\usepackage{graphicx}
\usepackage{wrapfig}
\usepackage{caption}
\usepackage{subcaption}
\usepackage{indentfirst}
\usepackage[left=2.5cm,right=2.5cm,top=2.5cm,bottom=2.5cm,bindingoffset=0cm]{geometry}
\usepackage{enumerate}
\usepackage{makecell}
\usepackage{float}
\usepackage[normalem]{ulem}
\usepackage{color}

\DeclareMathOperator{\aut}{Aut}

\DeclareMathOperator{\cyc}{Cyc}

\DeclareMathOperator{\GaL}{{\rm \Gamma}L}
\DeclareMathOperator{\PGaL}{{\rm P\Gamma}L}

\DeclareMathOperator{\GL}{GL}

\DeclareMathOperator{\inv}{Inv}

\DeclareMathOperator{\iso}{Iso}

\DeclareMathOperator{\orb}{Orb}

\DeclareMathOperator{\rk}{rk}

\DeclareMathOperator{\SL}{SL}

\DeclareMathOperator{\sym}{Sym}

\DeclareMathOperator{\poly}{poly}
\DeclareMathOperator{\alg}{Alg}

\DeclareMathOperator{\ind}{ind}

\DeclareMathOperator{\Hol}{Hol}

\DeclareMathOperator{\Char}{char}

\def\tm#1{\item[{\rm (#1)}]}

\makeatletter 
\def\@seccntformat#1{\csname the#1\endcsname. } 
\def\@biblabel#1{#1.}

\makeatother

\title{On separability of Tatra association schemes}

\author{Grigory Ryabov}
\address{Sobolev Institute of Mathematics, Novosibirsk, Russia}
\address{Novosibirsk State Technical University, Novosibirsk, Russia}
\email{gric2ryabov@gmail.com}

\thanks{The author was supported by the state contract of the Sobolev Institute of Mathematics (project number FWNF-2026-0011)}

\date{}

\newtheorem{prop}{Proposition}[section]

\newtheorem{lemm}[prop]{Lemma}
\newtheorem{theo}[prop]{Theorem}

\theoremstyle{definition}

\newtheorem*{rem}{Remark}

\begin{document}

\begin{abstract}
A Tatra association scheme is an association scheme arising from a symmetric bilinear form defined on the equivalence classes of nonzero $2$-dimensional vectors modulo some subgroup of the multiplicative group of a finite field. In the present paper, we prove that every such association scheme is $4$-separable, i.e., it is determined up to isomorphism by the tensor of its $4$-dimensional intersection numbers.
\\
\\
\textbf{Keywords}: association schemes, separability, linear groups.
\\
\\
\textbf{MSC}: 05E30, 20B25.
\end{abstract}

\maketitle

\section{Introduction}
A \emph{coherent configuration} $\mathcal{X}$ on a finite set $\Omega$ can be thought of as a special partition of $\Omega \times \Omega$ for which the diagonal of $\Omega \times \Omega$ is a union of some classes (see~\cite{CP} for the exact definition). The number $|\Omega|$ is called the \emph{degree} of $\mathcal{X}$. If the diagonal is exactly a class of the partition, then $\mathcal{X}$ is called an \emph{association scheme} or just a \emph{scheme}. To date, coherent configurations are realized as one of the main objects in algebraic combinatorics, used for studying permutation groups, combinatorial designs, codes, and isomorphisms of combinatorial objects. For a background of coherent configurations and recent progress in their study, we refer the reader to the monograph~\cite{CP}.

In the present paper, we are interested in a special class of schemes, namely, so-called \emph{Tatra schemes}. These schemes are constructed from a symmetric bilinear form defined on the equivalence classes of nonzero $2$-dimensional vectors modulo some subgroup of the multiplicative group of a finite field (see Section~3 for the exact definition). In fact, Tatra schemes generalize the schemes arising from the actions of $\SL(2,q)$ on pairs of nonzero elements of $\mathbb{F}_q^2$ and $\PGaL(2,q)$ on pairs of one-dimensional subspaces of $\mathbb{F}_q^2$. If $n=1$, then a Tatra scheme is trivial, whereas if $n\geq 2$, then every nonthin relation of a Tatra scheme is an edge set of a distance-regular antipodal graph of diameter~3 from~\cite[Theorem~12.5.3]{BCN}. A Tatra scheme depends on two parameters, namely, the order $q$ of the finite field and a divisor $n$ of $q-1$ such that $q(q-1)/n$ is even. The degree of a Tatra scheme is equal to $n(q+1)$.

Tatra schemes were introduced in~\cite{Reich} and go back to the questions from~\cite{Nevo} related to investigations of certain simplicial complexes. In~\cite{KhS}, such schemes were constructed independently from balanced generalized weighing matrices. Fusions and isomorphisms of Tatra schemes were studied in~\cite{MR}. In the same paper, Tatra schemes were used for constructing strongly regular digraphs and divisible design graphs. Noncommutative schemes of rank~6 from~\cite{DF} are fusions of Tatra schemes.

The main goal of the present paper is to give a combinatorial characterization of Tatra schemes. To do this, we use the notion of \emph{separability} of a coherent configuration. Informally, a coherent configuration is said to be \emph{$m$-separable} if it is determined up to isomorphism by the tensor of its $m$-dimensional intersection numbers. An isomorphism between an $m$-separable coherent configuration on $n$ points and any other coherent configuration can be verified in time $\poly(n^m)$. The \emph{separability number} $s(\mathcal{X})$ of $\mathcal{X}$ is defined to be the smallest positive integer $m$ for which $\mathcal{X}$ is $m$-separable. The concept of $m$-separability was introduced in~\cite{EP0}. Several results on $m$-separability of schemes can be found in~\cite{CHPV,CP0,EP0,EP,HKP,MP,PR,PV}. 

Let 
$$\mathcal{Q}_0=\{2^d:~6\leq d\leq 20,~d\neq 13,17,19\}\cup \{3^d:~d\in\{4,5,6,8,10\}\}\cup \{5^d:~d\in \{4,5,6\}\}.$$ The main result of this paper is the theorem below on separability of Tatra schemes.

\begin{theo}\label{main}
Let $\mathcal{X}=\mathcal{X}(q,n)$ be a Tatra association scheme. Then $s(\mathcal{X})\leq 4$. Moreover, the following statements hold:
\begin{enumerate}

\tm{1} $s(\mathcal{X})\leq 3$ if $q\notin \mathcal{Q}_0$;

\tm{2} $s(\mathcal{X})\leq 2$ if $n=q-1$;

\tm{3} $s(\mathcal{X})>1$ unless $n\leq 2$ or $\Char(\mathbb{F}_q)$ is a primitive root modulo $n$.

\end{enumerate}
\end{theo}

To prove an upper bound on $s(\mathcal{X})$, we study a one-point extension of $\mathcal{X}$. Namely, we show that a restriction of this extension to a fiber is a cyclotomic scheme (see Section~2 for the definition). Further, we use the results from~\cite{P} on separability of cyclotomic schemes. The condition $q\notin \mathcal{Q}_0$ in Statement~$(1)$ of Theorem~\ref{main} appears from the result on a separability number of a cyclotomic scheme~\cite[Theorem~1.1]{P} stating that a cyclotomic scheme of degree~$q$ is $2$-separable unless $q\in \mathcal{Q}_0$. As it was mentioned by I. Ponomarenko in a private communication, if $q\in \mathcal{Q}_0$, then a cyclotomic scheme of degree~$q$ is also $2$-separable, but this result is unpublished now. Using the latter unpublished result, it is possible to deduce that $s(\mathcal{X})\leq 3$ for every Tatra scheme $\mathcal{X}$.

To prove a lower bound on $s(\mathcal{X})$, we compute the group of algebraic automorphisms of $\mathcal{X}$ (Lemma~\ref{algiso}) and use the results on combinatorial isomorphisms of $\mathcal{X}$ (Lemma~\ref{combiso}). Observe that if $n\leq 2$ or $\Char(\mathbb{F}_q)$ is a primitive root modulo $n$, then the Tatra scheme $\mathcal{X}(q,n)$ can be $1$-separable. If $n=1$, then $\mathcal{X}$ is trivial and hence $1$-separable. If $q=5$ and $n=2$, then the Tatra scheme $\mathcal{X}(5,2)$ of degree~$12$ is $1$-separable because every scheme of degree at most~$14$ is $1$-separable (see~\cite[Section~2.3.4]{CP}). The smallest noncommutative Tatra scheme $\mathcal{X}(4,3)$ of degree~$15$ is $1$-separable. This can be easily verified using the list of small schemes~\cite{HM}. In this case, $\Char(\mathbb{F}_q)=2$ is a primitive root modulo $n=3$.

We finish the introduction with a brief outline of the paper. Section~2 contains a necessary background of coherent configurations, their isomorphisms, point extensions, and cyclotomic schemes. In Section~3, we provide information on Tatra schemes. In particular, we study their algebraic automorphisms (Section~3.3). In Section~4, we prove Theorem~\ref{main}.

The author would like to thank the anonymous referee for the comments which help to improve the text.

\section{Coherent configurations}

Throughout the text, we freely use basic definitions and facts from the theory of coherent configurations and recall in this section some of them which will be crucial for further explanation. For a background of coherent configurations, we refer the reader to~\cite{CP}, where the terminology used in the paper and all facts not explained in detail are contained. A composition of two binary relations $r$ and $s$ on the same set is denoted by $rs$. The relation inverse to a binary relation $s$ is denoted by $s^*$. The equivalence closure of a binary relation $s$, i.e., the smallest with respect to inclusion equivalence relation on the same set, is denoted by $\langle s \rangle$. Given a set $X$, the set of all possible unions of some elements from $X$ is denoted by $X^\cup$.

\subsection{Notation and basic facts}

Let $\mathcal{X}=(\Omega,S)$ be a coherent configuration on a finite set $\Omega$ and let $S$ be the set of all basis relations of $\mathcal{X}$. If the diagonal $1_{\Omega}$ of $\Omega\times \Omega$ is a basis relation of $\mathcal{X}$, then $\mathcal{X}$ is an \emph{association scheme} or just a \emph{scheme}. Every binary relation from $S^\cup$ is defined to be a \emph{relation} of $\mathcal{X}$. If $s$ is a relation of $\mathcal{X}$, then $\langle s \rangle$ is a \emph{parabolic} of $\mathcal{X}$.

We say that $\mathcal{X}$ is \emph{trivial} if $S=\{1_{\Omega},\Omega^2\setminus 1_{\Omega}\}$. Clearly, $\mathcal{X}$ is a scheme in this case.

Given $r,s,t\in S$, the corresponding intersection number of $\mathcal{X}$ is denoted by $c_{rs}^t$. A coherent configuration is said to be \emph{commutative} if $c_{rs}^t=c_{sr}^t$ for all $r,s,t\in S$. Every commutative coherent configuration is a scheme.

If $s\in S^\cup$ is such that $s\subseteq \Delta\times \Delta^\prime$ for some $\Delta,\Delta^\prime\in F(\mathcal{X})$ and $\alpha\in \Delta$, then $n_s$ denotes the valency of $s$ and $\alpha s=\{\beta\in \Delta^\prime:~(\alpha,\beta)\in s\}$. Clearly, $n_s=|\alpha s|$ for every $\alpha \in \Delta$.

A coherent configuration $\mathcal{X}$ is said to be \emph{partly regular} if there exists $\alpha\in \Omega$ such that $|\alpha s|\leq 1$ for every $s\in S$. An association scheme $\mathcal{X}$ such that $n_s=1$ for every $s\in S$ is said to be \emph{regular} or \emph{thin}. Clearly, if a scheme $\mathcal{X}$ is regular, then $\mathcal{X}$ is partly regular.

Given $\Delta,\Delta^\prime\subseteq\Omega$, put 
$$S_{\Delta,\Delta^\prime}=\{s_{\Delta\times\Delta^\prime}:~s\in S,~s_{\Delta\times\Delta^\prime}\neq \varnothing\},$$
where $s_{\Delta\times\Delta^\prime}=s\cap (\Delta\times\Delta^\prime)$, $s_{\Delta}=s_{\Delta,\Delta}$, and $S_{\Delta}=S_{\Delta,\Delta}$.

The set of all fibers of $\mathcal{X}$ is denoted by $F(\mathcal{X})$. If $\Delta,\Delta^\prime\in F(\mathcal{X})^\cup$, then $\Delta\times\Delta^\prime\in S^\cup$ and the pair 
$$\mathcal{X}_{\Delta}=(\Delta,S_{\Delta})$$
is a coherent configuration on $\Delta$. If $\Delta\in F(\mathcal{X})$, then $\mathcal{X}_{\Delta}$ is a scheme.

Given a coherent configuration $\mathcal{X}^\prime$ on $\Omega$, put 
$$\mathcal{X}^\prime\geq \mathcal{X}$$ 
if every basis relation of $\mathcal{X}$ is a relation of $\mathcal{X}^\prime$. In this case, we say that $\mathcal{X}^\prime$ is a \emph{fission} of $\mathcal{X}$. One can see that every fiber of $\mathcal{X}$ is a union of some fibers of $\mathcal{X}^\prime$. Observe also that if $\Delta,\Delta^\prime\in F(\mathcal{X}^\prime)$ and $s$ is a basis relation of $\mathcal{X}$, then $s\cap (\Delta\times \Delta^\prime)$ is a relation of $\mathcal{X}^\prime$. Clearly, if $\mathcal{X}$ is partly regular, then $\mathcal{X}^\prime$ is also partly regular.

\subsection{Isomorphisms and separability}

Let $\mathcal{X}=(\Omega,S)$ and $\mathcal{X}^{\prime}=(\Omega^{\prime},S^{\prime})$ be coherent configurations. An \emph{algebraic isomorphism} from $\mathcal{X}$ to $\mathcal{X}^{\prime}$ is defined to be a bijection $\varphi:S\rightarrow S^{\prime}$ such that
$$c_{rs}^t=c_{r^\varphi,s^\varphi}^{t^\varphi}$$
for all $r,s,t\in S$. In this case, we say that $\mathcal{X}$ and $\mathcal{X}^\prime$ are \emph{algebraically isomorphic} and we have
\begin{equation}\label{isoalgprop1}
n_{s^\varphi}=n_s
\end{equation}
and 
\begin{equation}\label{isoalgprop2}
\langle s\rangle^\varphi=\langle s^\varphi \rangle
\end{equation}
for every $s\in S$ by~\cite[Corollary~2.3.20]{CP} and~\cite[Exercise~2.7.29(2)]{CP}, respectively. The set of all algebraic isomorphisms from $\mathcal{X}$ to itself (algebraic automorphisms) forms a subgroup of $\sym(S)$ denoted by $\aut_{\alg}(\mathcal{X})$.

A (\emph{combinatorial}) \emph{isomorphism} from $\mathcal{X}$ to $\mathcal{X}^{\prime}$ is defined to be a bijection $f:\Omega\rightarrow \Omega^{\prime}$ such that $S^{\prime}=S^f$, where $S^f=\{s^f:~s\in S\}$ and $s^f=\{(\alpha^f,\beta^f):~(\alpha,~\beta)\in s\}$. In this case, we say that $\mathcal{X}$ and $\mathcal{X}^\prime$ are \emph{isomorphic}. The group $\iso(\mathcal{X})$ of all isomorphisms from $\mathcal{X}$ onto itself has a normal subgroup
$$\aut(\mathcal{X})=\{f\in \iso(\mathcal{X}): s^f=s~\text{for every}~s\in S\}$$
called the \emph{automorphism group} of $\mathcal{X}$.

If $G\leq \sym(\Omega)$, then the pair $\inv(G)=(\Omega,\orb(G,\Omega^2))$, where $\orb(G,\Omega^2)$ is the set of all orbits of $G$ acting on $\Omega^2$ componentwise, is a coherent configuration. Given $G_1,G_2\leq \sym(\Omega)$, we have if $G_1\leq G_2$, then $\inv(G_1)\geq \inv(G_2)$. A coherent configuration $\mathcal{X}$ is said to be \emph{schurian} if $\mathcal{X}=\inv(G)$ for some $G\leq \sym(\Omega)$. One can verify that $\inv(\aut(\mathcal{X}))\geq \mathcal{X}$ for every coherent configuration $\mathcal{X}$ and the equality is attained if and only if $\mathcal{X}$ is schurian.

Every isomorphism between coherent configurations induces in a natural way an algebraic isomorphism between them. However, not every algebraic isomorphism is induced by a combinatorial one. A coherent configuration is said to be \emph{separable} if every algebraic isomorphism from it to any coherent configuration is induced by an isomorphism. It is easy to see that the trivial coherent configuration is separable. The subgroup of $\aut_{\alg}(\mathcal{X})$ consisting of all algebraic automorphisms induced by the elements of $\iso(\mathcal{X})$ is denoted by $\aut^{\ind}_{\alg}(\mathcal{X})$.

It is easy to see that $f\in \iso(\mathcal{X})$ induces the trivial algebraic automorphism of $\mathcal{X}$ which fixes every basis relation if and only if $f\in \aut(\mathcal{X})$ and hence $f_1,f_2\in \iso(\mathcal{X})$ induce the same algebraic isomorphism if and only if $f_2f_1^{-1}\in \aut(\mathcal{X})$. Therefore
\begin{equation}\label{algind}
|\aut^{\ind}_{\alg}(\mathcal{X})|=\frac{|\iso(\mathcal{X})|}{|\aut(\mathcal{X})|}.
\end{equation}

\begin{lemm}\cite[Theorem~3.3.19]{CP}\label{regularsep}
Every partly regular coherent configuration is schurian and separable.
\end{lemm}

\begin{lemm}\cite[Lemma~9.2]{EP}\label{deletefiber}
Let $\mathcal{X}=(\Omega,S)$ be a coherent configuration and let $\Delta\in F(\mathcal{X})^\cup$ be such that for every $\Delta^\prime\in F(\mathcal{X})$ with $\Delta^\prime\not\subseteq \Delta$, there is a basis relation $r\subseteq \Delta\times \Delta^\prime$ of valency~$1$. Then $\mathcal{X}$ is separable whenever so is $\mathcal{X}_{\Delta}$.
\end{lemm}

Let $m\geq 1$. The $m$-extension of a coherent configuration $\mathcal{X}$ on $\Omega$ is defined to be the smallest coherent configuration on $\Omega^m$ which contains the Cartesian $m$-power of $\mathcal{X}$ and for which the diagonal of $\Omega^m$ is the union of fibers. The intersection numbers of the $m$-extension are called the \emph{$m$-dimensional intersection numbers} of the coherent configuration $\mathcal{X}$. If $m=1$, then the $m$-extension of $\mathcal{X}$ coincides with $\mathcal{X}$ and the $m$-dimensional intersection numbers of $\mathcal{X}$ are the ordinary intersection numbers.

An algebraic isomorphism $\varphi$ from $\mathcal{X}$ to $\mathcal{X}^\prime$ is said to be \emph{$m$-dimensional} if it can be extended to an algebraic isomorphism from the $m$-extension of $\mathcal{X}$ to that of $\mathcal{X}^\prime$ that takes the diagonal of $\Omega^m$ to the diagonal of $(\Omega^\prime)^m$. A coherent configuration $\mathcal{X}$ is said to be \emph{$m$-separable} if every $m$-dimensional algebraic isomorphism from $\mathcal{X}$ is induced by some (combinatorial) isomorphism. The \emph{separability number} $s(\mathcal{X})$ of $\mathcal{X}$ is defined to be the smallest positive integer $m$ for which $\mathcal{X}$ is $m$-separable. It is easy to check that $1\leq s(\mathcal{X})\leq |\Omega|$ (see also~\cite[Eq.~4.2.1]{CP}). Due to the definitions, $s(\mathcal{X})=1$ if and only if $\mathcal{X}$ is separable. The equality $s(\mathcal{X})=m$ expresses the fact that $\mathcal{X}$ is determined up to isomorphism by the tensor of its $m$-dimensional intersection numbers.

\subsection{Point extensions}

Given $k\geq 1$ and $\alpha_1,\ldots,\alpha_k \in \Omega$, the \emph{$k$-point extension} $\mathcal{X}_{\alpha_1,\ldots,\alpha_k}$ of $\mathcal{X}$ with respect to $\alpha_1,\ldots,\alpha_k$ is defined to be the smallest coherent configuration on $\Omega$ such that $\mathcal{X}_{\alpha_1,\ldots,\alpha_k}\geq \mathcal{X}$ and $\{\alpha_1\},\ldots,\{\alpha_k\}\in F(\mathcal{X}_{\alpha_1,\ldots,\alpha_k})$. Observe that the ordering of the points $\alpha_1,\ldots,\alpha_k$ is unessential.

The first statement of the lemma below is~\cite[Lemma~3.3.5(2)]{CP}, the second one is~\cite[Proposition~3.3.3(1)]{CP}, and the third one is~\cite[Theorem~3.3.7(1)]{CP}. 

\begin{lemm}\label{onepoint}
Let $\mathcal{X}=(\Omega,S)$ be a coherent configuration and $\alpha\in \Omega$.
\begin{enumerate}

\tm{1} If $r\cap (\alpha s\times \alpha t)\neq \varnothing$, then $r\cap (\alpha s\times \alpha t)$ is a relation of $\mathcal{X}_{\alpha}$ for all $r,s,t\in S$. 

\tm{2} $\aut(\mathcal{X}_{\alpha})=\aut(\mathcal{X})_{\alpha}$.

\tm{3} If $\mathcal{X}$ is schurian, then $F(\mathcal{X}_{\alpha})=\{\alpha s:~s\in S\}$. 

\end{enumerate}
\end{lemm}

\begin{lemm}\cite[Theorem~4.6(1)]{EP0}\label{s1}
Let $\mathcal{X}$ be a coherent configuration. Then $s(\mathcal{X})\leq s(\mathcal{X}_{\alpha})+1$ for every point $\alpha$ of $\mathcal{X}$.
\end{lemm}

\subsection{Cyclotomic schemes}

Let $\mathbb{F}$ be a finite field of order~$q$ and let $K$ be a subgroup of $\mathbb{F}^*$ of index~$n$ for some divisor $n$ of $q-1$, where $\mathbb{F}^*$ is the multiplicative group of $\mathbb{F}$. Given $x\in \mathbb{F}^*/K$, put $s_x=\{(\alpha,\beta)\in \mathbb{F}^2:~\beta-\alpha\in x\}$. The pair $\mathcal{X}=\cyc(q,n)=(\mathbb{F},\{1_{\mathbb{F}},s_x:~x\in \mathbb{F}^*/K\})$ is a scheme called \emph{cyclotomic} (see, e.g.,~\cite[Section~2.2.2]{CP}). One can see that a valency of every irreflexive basis relation of $\mathcal{X}$ is equal to $(q-1)/n$. Note that if  $K=\mathbb{F}_q^{*}$, then $\mathcal{X}$ is trivial, whereas if $K$ is trivial, then $\mathcal{X}$ is regular.

\begin{lemm}\cite[p.~34]{CP}\label{cyclschurian}
With the above notation, let $G=\{f_{k,b}:~k\in K,~b\in \mathbb{F}\}$, where $f_{k,b}$ is a permutation on $\mathbb{F}$ such that $x^{f_{k,b}}=kx+b$ for every $x\in \mathbb{F}$. Then $\cyc(q,n)=\inv(G)$. In particular, every cyclotomic scheme is schurian.
\end{lemm}

Below, we provide a statement on separability of extensions of cyclotomic schemes which is required for the proof of Theorem~\ref{main}.

\begin{lemm}\label{cyclsepar}
Every fission of every two-point extension of a nontrivial cyclotomic scheme $\mathcal{X}$ is separable. Moreover, if the degree of $\mathcal{X}$ does not belong to $\mathcal{Q}_0$, then every fission of every one-point extension of $\mathcal{X}$ is separable.
\end{lemm}

\begin{proof}
Let us prove the first part of the lemma. Let $\alpha,\beta\in \Omega=\mathbb{F}$ with $\alpha\neq\beta$ and $\mathcal{Y}=\mathcal{X}_{\alpha,\beta}$. Let $\mathcal{Y}^\prime$ be a coherent configuration on $\Omega$ such that $\mathcal{Y}^\prime\geq \mathcal{Y}$. Clearly, $\{\alpha\}\in F(\mathcal{Y}^\prime)$. From~\cite[Lemma~4.4.6, Theorem~4.4.7]{CP} it follows that $(\mathcal{X}_{\alpha})_{\Omega\setminus \{\alpha\}}\geq \mathcal{Z}$ for some scheme $\mathcal{Z}$ whose every point extension is partly regular. Using the definition of a point extension, one can deduce that 
$$\mathcal{Y}^\prime_{\Omega\setminus \{\alpha\}}\geq \mathcal{Y}_{\Omega\setminus \{\alpha\}}\geq ((\mathcal{X}_{\alpha})_{\Omega\setminus \{\alpha\}})_{\beta}\geq \mathcal{Z}_{\beta}$$
and hence $\mathcal{Y}^\prime_{\Omega\setminus \{\alpha\}}$ is partly regular. So $\mathcal{Y}^\prime_{\Omega\setminus \{\alpha\}}$ is separable by Lemma~\ref{regularsep}. Therefore $\mathcal{Y}^\prime$ is separable by~\cite[Lemma~2.5]{P} and we are done.

Now let us prove the second part of the lemma. Suppose that $q\notin \mathcal{Q}_0$. Due to~\cite[Proof of Theorem~1.1]{P}, there exists $\alpha\in \Omega$ such that every fission of $(\mathcal{X}_{\alpha})_{\Omega\setminus \{\alpha\}}$ is separable. So every fission of $\mathcal{X}_{\alpha}$ is separable by~\cite[Lemma~2.5]{P}. Since $\mathcal{X}$ is schurian, all one-point extensions of $\mathcal{X}$ are pairwise algebraically isomorphic by~\cite[Theorem~3.3.7(2)]{CP}. Together with separability of $\mathcal{X}_{\alpha}$, this implies that all one-point extensions of $\mathcal{X}$ are pairwise isomorphic. Therefore every fission of every one-point extension of $\mathcal{X}$ is separable.
\end{proof}

\section{Tatra schemes}

In this section, we recall the construction of Tatra schemes and several of their properties required for the proof of Theorem~\ref{main}. 

\subsection{Construction and basic properties}
Let $r$ be a prime, $q=r^d$ for some $d\geq 1$, and $\mathbb{F}=\mathbb{F}_q$ a field of order~$q$. Let $n$ be a divisor of $q-1$ such that 
$$\frac{q(q-1)}{n}~\text{is even},$$ 
$K$ a subgroup of $\mathbb{F}^*$ of index~$n$, and $C=\mathbb{F}^*/K$. Clearly, $C\cong C_n$ and $m=|K|=(q-1)/n$ is even whenever $q$ is odd and hence $(-1)\in K$. We denote the identity element of $C$ by~$e$. If $x\in\mathbb{F}$ and $g=Ky\in C$, then put $xg=K(xy)$.

Let $V=\mathbb{F}^2$ and 
$$\Omega=\{Kv:~v\in V\setminus \{0\}\},$$
where $Kv=\{xv:~x\in K\}$.
Given $\alpha=Kv\in \Omega$ and $g=Kx\in C$, put $g\alpha=(Kx)(Kv)=K(xv)$. One can see that 
$$|\Omega|=|V\setminus \{0\}|/|K|=(q^2-1)/m=(q+1)n.$$ 
Given vectors $u=(u_1,u_2)^T$ and $v=(v_1,v_2)^T$ from $V$, put
$$\det(u,v)=u_1v_2-u_2v_1.$$
Let
$$\langle\cdot,~\cdot \rangle: \Omega\times \Omega \rightarrow C\cup \{0\}$$
be the form defined as follows:
$$\langle Ku,Kv\rangle=K\det(u,v).$$ 
It is easy to see that $\langle Ku,Kv\rangle=0$ if and only if $Kv=gKu$ for some $g\in C$. The form $\langle\cdot,~\cdot \rangle$ is well-defined by~\cite[Lemma~3]{Reich}. It is also symmetric. Indeed, this is clear if $q$ is even. If $q$ is odd, then $m$ is even and hence $(-1)\in K$ which implies the symmetry of $\langle\cdot,~\cdot \rangle$.

Given $g\in C$, let us define two binary relations $r_g$ and $s_g$ on $\Omega^2$ as follows:
$$r_g=\{(\alpha,\beta)\in \Omega^2:~\langle \alpha,\beta \rangle=0,~\beta=g\alpha\},$$
$$s_g=\{(\alpha,\beta)\in \Omega^2:~\langle \alpha,\beta \rangle=g\}.$$
Observe that $r_e=1_{\Omega}$. Due to~\cite[Theorem~3]{Reich}, the pair  
$$\mathcal{X}=(\Omega,S),~\text{where}~S=\{r_g,s_g:~g\in C\}$$
is a scheme. This scheme $\mathcal{X}$ is defined to be a \emph{Tatra scheme} with parameters $q$ and~$n$. Clearly, $\rk(\mathcal{X})=2|C|=2n$.

\begin{lemm}\cite[Proposition~3.1]{MR}\label{scheme}
With the above notation, 
\begin{enumerate}
\tm{1} $n_{r_g}=1$, $n_{s_g}=q$,

\tm{2} $r_g^*=r_{g^{-1}}$, $s_g^*=s_g$

\tm{3} $r_hr_g=r_gr_h=r_{hg}$,

\tm{4} $r_hs_g=s_{h^{-1}g}$, $s_gr_h=s_{gh}$,

\tm{5} $c_{s_hs_g}^{r_{h^{-1}g}}=q$,~$c_{s_hs_g}^{r_{x}}=0$,~$c_{s_hs_g}^{s_{y}}=m$ 

\end{enumerate}
for all $h,g,y\in C$ and all $x\in C\setminus \{h^{-1}g\}$.
\end{lemm}

Note that if $n=1$, then $\mathcal{X}$ is trivial, if $n=2$, then $\mathcal{X}$ is a commutative scheme of rank~$4$, and if $n\geq 3$, then $\mathcal{X}$ is noncommutative.

Put
$$r_C=\bigcup \limits_{g\in C} r_g.$$
One can see that $r_C$ is an equivalence relation on $\Omega$ consisting of all pairs $(\alpha,\beta)\in \Omega^2$ such that $\beta=g\alpha$ for some $g\in C$. Therefore $r_C$ is a parabolic of $\mathcal{X}$. Clearly, $n_{r_C}=|C|=n$ and $|\Omega/r_C|=|\Omega|/n_{r_C}=q+1$, where $\Omega/r_C$ denotes the set of all equivalence classes of $r_C$.

\begin{lemm}\label{distinct}
With the above notation, $|\alpha s_g\cap \Gamma|=1$ for all $\alpha\in \Omega$, $g\in C$, and $\Gamma\in \Omega/r_C$ such that $\alpha\notin \Gamma$.
\end{lemm}

\begin{proof}
Assume that $|\alpha s_g\cap \Gamma|\geq 2$ for some $\alpha\in \Omega$, $g\in C$, and $\Gamma\in \Omega/r_C$ such that $\alpha\notin \Gamma$. Then there exist distinct $\beta,\gamma\in \Gamma$ such that $(\alpha,\beta),(\alpha,\gamma)\in s_g$. This implies that $c_{s_g^*s_g}^r\geq 1$, where $r\in S$ is such that $(\beta,\gamma)\in r$. Since distinct $\beta$ and $\gamma$ lie in the same $r_C$-class $\Gamma$, we conclude that $r=r_h$ for some $h\in C\setminus \{e\}$. Due to Lemma~\ref{scheme}(2), we have $s_g^*=s_g$. Therefore
$$c_{s_gs_g}^{r_h}=c_{s_g^*s_g}^r\geq 1,$$
a contradiction to Lemma~\ref{scheme}(5). Thus, $|\alpha s_g\cap \Gamma|\leq 1$ for all $\alpha\in \Omega$, $g\in C$, and $\Gamma\in \Omega/r_C$ such that $\alpha\notin \Gamma$. Together with $n_{s_g}=q$, $g\in C$, (Lemma~\ref{scheme}(1)) and $|\Omega/r_C|=q+1$, this yields the required. 
\end{proof}

\subsection{Combinatorial isomorphisms}

Let $\Sigma=\aut(\mathbb{F}_q)$. The group $\GaL(2,q)=\GL(2,q)\rtimes \Sigma$ acts on $\Omega$ as follows: if $T\sigma\in \GaL(2,q)$, where $T\in \GL(2,q)$ and $\sigma\in \Sigma$, and $Kv\in \Omega$, then
$$(Kv)^{T\sigma}=K(Tv^\sigma),$$
where $\sigma$ acts on each component of~$v$. The kernel of this action is the group $\widetilde{K}=\{xI_2:~x\in K\}\cong K$ of order $m=|K|$, where $I_2$ is the identity $(2\times 2)$-matrix. 

Put
$$\GL(2,q)_K=\{T\in \GL(2,q):~\det(T)\in K\}.$$
Clearly, $\GL(2,q)_K$ is a subgroup of $\GL(2,q)$ and $\GL(2,q)_K$ induces the action on $\Omega$ and the kernel of this action is $\widetilde{K}$. Observe that $|\GL(2,q)_K|=|\SL(2,q)||K|$.

The group $\Sigma$ induces the action by automorphisms on $C=\mathbb{F}^*/K$ in the following way: if $\sigma \in \Sigma$ and $g=Kx\in C$, then
$$g^\sigma=(Kx)^\sigma=K(x^\sigma).$$ 
The kernel of the action of $\Sigma$ on $C$ is denoted by $\Sigma_0$. Observe that $\Sigma_0$ can act nontrivially on $\Omega$ (see~\cite{MR}).

Lemma below collects necessary facts on combinatorial isomorphisms of $\mathcal{X}$.

\begin{lemm}\label{combiso}
With the above notation, let $n\geq 2$. Then the following statements hold.
\begin{enumerate}

\tm{1} If $f=(T\sigma)^\Omega\in \GaL(2,q)^\Omega$, where $T\in \GL(2,q)$ and $\sigma\in \Sigma$, then $r_g^f=r_{g^\sigma}$ and $s_g^f=s_{\det(T)g^\sigma}$ for every $g\in C$.

\tm{2} $\aut(\mathcal{X})\geq (\GL(2,q)_K\rtimes \Sigma_0)^\Omega$ and the equality is attained if $n\geq 3$.

\tm{3} $\iso(\mathcal{X})\geq \GaL(2,q)^{\Omega}$ and the equality is attained if $n\geq 3$.

\tm{4} $\mathcal{X}$ is schurian.

\end{enumerate}
\end{lemm}

\begin{proof}
Statement~$(1)$ in case $n\geq 3$ holds by~\cite[Lemma~3.11]{MR}. However, the proof remains the same in case $n=2$. The inclusions $\aut(\mathcal{X})\geq (\GL(2,q)_K\rtimes \Sigma_0)^\Omega$ and $\iso(\mathcal{X})\geq \GaL(2,q)^{\Omega}$ are deduced from Statement~$(1)$ in case $n\geq 3$ in~\cite[Theorem~3.6]{MR}. Again, the proofs of these inclusions remain the same when $n=2$. The condition $n\geq 3$ is essential only in the proofs of the equalities $\aut(\mathcal{X})=(\GL(2,q)_K\rtimes \Sigma_0)^\Omega$ and $\iso(\mathcal{X})=\GaL(2,q)^{\Omega}$. Statement~(4) is deduced from Statement~$(2)$ in case $n\geq 3$ in~\cite[Corollary~3.17]{MR} and the proof is valid also in case $n=2$. 
\end{proof}

\subsection{Algebraic isomorphisms}

In this subsection, we assume that $n\geq 3$. Let us define the binary operation $\star$ on $S$ as follows. If one of the relations $r,s\in S$ is of valency~$1$, then put $r\star s=rs$. Otherwise, $r=s_{h}$ and $s=s_{g}$ for some $h,g\in C$. In this case, put $r\star s=r_{h^{-1}g}$. Clearly, $r_e=1_\Omega$ is an identity element with respect to $\star$ and $s\star s^*=s^*\star s=r_e$ for every $s\in S$. From Statements~(3)-(5) of Lemma~\ref{scheme} it follows that the set $S$ equipped with the binary operation $\star$ is a group isomorphic to a dihedral group $D_{2n}$ whose cyclic subgroup of index~$2$ is $\{r_g:~g\in C\}$. Denote this group by $S^\star$. Since $S^\star\cong D_{2n}$, we have
$$\aut(S^\star)=\{\varphi_{\sigma,g}:~\sigma\in \aut(C),~g\in C\}\cong \aut(D_{2n})\cong \Hol(C),$$
where $\varphi_{\sigma,g}\in \sym(S)$ is such that
\begin{equation}\label{defnf}
r_h^{\varphi_{\sigma,g}}=r_{h^\sigma}~\text{and}~s_h^{\varphi_{\sigma,g}}=s_{h^\sigma g}
\end{equation}
for every $h\in C$.

\begin{lemm}\label{algiso}
With the above notation, $\aut_{\alg}(\mathcal{X})=\aut(S^\star)$.
\end{lemm}

\begin{proof}
Using Statements~$(3)$-$(5)$ of Lemma~\ref{scheme}, it can be verified directly that $c_{r^\varphi s^\varphi}^{t^\varphi}=c_{rs}^t$ for all $r,s,t\in S$ and $\varphi=\varphi_{\sigma,g}$, $\sigma\in \aut(C)$, $g\in C$. For example, let $r=s_x$, $s=s_y$, and $t=r_{x^{-1}y}$ for some $x,y\in C$. Then $r^\varphi=s_{x^\sigma g}$, $s^\varphi=s_{y^\sigma g}$, and $t^\varphi=r_{(x^\sigma)^{-1}y^\sigma}$. So 
$$c_{r^\varphi s^\varphi}^{t^\varphi}=q=c_{rs}^t$$
by Lemma~\ref{scheme}(5). Thus, $\aut_{\alg}(\mathcal{X})\geq\aut(S^\star)$. 

Conversely, let $\varphi\in \aut_{\alg}(\mathcal{X})$. Let $g_0$ be a generator of $C$ and $r=r_{g_0}^\varphi$. Eq.~\eqref{isoalgprop1} and Lemma~\ref{scheme}(1) imply that $n_r=n_{r_{g_0}}=1$. So $r=r_{g_0^\prime}$ for some $g_0^\prime\in C$. Since $g_0$ is a generator of $C$, we conclude that $\langle r_{g_0} \rangle=r_C$. Together with Eq.~\eqref{isoalgprop2}, this yields that $\langle r_{g_0^\prime} \rangle=r_C$ and hence $g_0^\prime$ is a generator of $C$. Therefore there exists $\sigma\in \aut(C)$ such that $g_0^\sigma=g_0^\prime$ and hence
\begin{equation}\label{alg1}
r_{g_0}^\varphi=r_{g_0^\sigma}.
\end{equation}

Let $s=s_e^\varphi$. Due to Eq.~\eqref{isoalgprop1} and Lemma~\ref{scheme}(1), we have $n_s=n_{s_e}=q$. Therefore $s=s_g$ and hence
\begin{equation}\label{alg2}
s_e^\varphi=s_g
\end{equation}
for some $g\in C$. 

Let us prove that $\varphi=\varphi_{\sigma,g}$. At first, let us show that
\begin{equation}\label{alg3} 
r_{g_0^i}^\varphi=r_{(g_0^i)^\sigma}=r_{g_0^i}^{\varphi_{\sigma,g}}
\end{equation}
for every $i\geq 1$. We proceed by induction on $i$. This follows from Eq.~\eqref{alg1} in case $i=1$. If $i\geq 2$, then
$$r_{g_0^i}^\varphi=(r_{g_0^{i-1}}r_{g_0})^\varphi=r_{{g_0}^{i-1}}^\varphi r_{g_0}^\varphi=r_{({g_0}^{i-1})^\sigma}r_{{g_0}^\sigma}=r_{({g_0}^{i-1})^\sigma g_0^\sigma}=r_{(g_0^i)^\sigma}$$
as desired, where the first and fourth equalities hold by Lemma~\ref{scheme}(3), the second equality holds because $\varphi$ is an algebraic isomorphism, the third equality holds by the induction hypothesis, and the fifth one holds because $\sigma\in \aut(C)$.

Now let us show that 
\begin{equation}\label{alg4}
s_h^\varphi=s_{h^\sigma g}=s_h^{\varphi_{\sigma,g}}
\end{equation}
for every $h\in C$. Indeed,
$$s_h^\varphi=(s_er_h)^\varphi=s_e^\varphi r_h^\varphi=s_gr_{h^\sigma}=s_{h^\sigma g}$$
as desired, where the first and fourth equalities hold by Lemma~\ref{scheme}(4), the second one holds because $\varphi$ is an algebraic isomorphism, and the third one holds by Eqs.~\eqref{alg2} and~\eqref{alg3}. Therefore $\varphi=\varphi_{\sigma,g}$ by Eqs.~\eqref{alg3} and~\eqref{alg4}. Thus, $\aut_{\alg}(\mathcal{X})\leq\aut(S^\star)$ and hence $\aut_{\alg}(\mathcal{X})=\aut(S^\star)$ as required. 
\end{proof}

The lemma below establishes the proportion between the number of all algebraic automorphisms of $\mathcal{X}$ and the number of those of them which are induced by combinatorial isomorphisms. 

\begin{lemm}\label{compare}
With the above notation, 
$$\frac{|\aut_{\alg}(\mathcal{X})|}{|\aut^{\ind}_{\alg}(\mathcal{X})|}=\frac{\phi(n)d_0}{d},$$
where $d_0=|\Sigma_0|$ and $\phi$ is the Euler function.
\end{lemm}

\begin{proof}
Lemma~\ref{algiso} implies that 
$$|\aut_{\alg}(\mathcal{X})|=|\Hol(C)|=n\phi(n).$$ 
One can compute $|\aut^{\ind}_{\alg}(\mathcal{X})|$ as follows:
$$|\aut^{\ind}_{\alg}(\mathcal{X})|=\frac{|\iso(\mathcal{X})|}{|\aut(\mathcal{X})|}=\frac{|\GL(2,q)||\Sigma|}{|K||\SL(2,q)||\Sigma_0|}=\frac{(q-1)|\Sigma|}{m|\Sigma_0|}=\frac{nd}{d_0},$$
where the first equality holds by Eq.~\eqref{algind}, whereas the second one by Statements~$(2)$ and~$(3)$ of Lemma~\ref{combiso}. Now the required easily follows from the above two equalities. 
\end{proof}

\begin{rem}
By the estimates for the Euler function (see, e.g.,~\cite[Theorem~15]{RS}) and $d/d_0\leq \log(q)$, we have
$$\frac{\phi(n)d_0}{d}\geq \frac{n}{2c_0\log(q)\log(\log(n))}$$
for some constant $c_0>0$. If $n\geq c_1q^{c_2}$ for some constants $c_1,c_2>0$, then 
$$|\aut_{\alg}(\mathcal{X})|/|\aut^{\ind}_{\alg}(\mathcal{X})|\underset{q \to \infty}{\longrightarrow} \infty.$$
\end{rem}

\section{Proof of Theorem~\ref{main}}

Throughout this section, we keep the notation from the previous one. If $n=1$, then $\mathcal{X}$ is trivial and hence separable. So $s(\mathcal{X})=1$ in this case. Further we assume that $n\geq 2$. At first, let us prove that $s(\mathcal{X})\leq 4$ and $s(\mathcal{X})\leq 3$ whenever $q\notin \mathcal{Q}_0$.

Let $\alpha=K(1,0)\in \Omega$, $\mathcal{Y}=\mathcal{X}_{\alpha}$, and let $T$ be the set of all basis relations of $\mathcal{Y}$. Since $\mathcal{X}$ is schurian (Lemma~\ref{combiso}(4)), we conclude that
\begin{equation}\label{fibers}
F(\mathcal{Y})=\{\alpha s:~s\in S\}
\end{equation}
by Lemma~\ref{onepoint}(3). Let $\Delta=\alpha s_e$. By the definition of $s_e$, we have 
\begin{equation}\label{deltaset}
\Delta=\{K(x,1)^T:~x\in \mathbb{F}\}.
\end{equation}

\begin{lemm}\label{fibere}
For every $\Delta^\prime\in F(\mathcal{Y})$ with $\Delta^\prime\neq \Delta$, there is $t\in T_{\Delta,\Delta^\prime}$ such that $n_t=1$.
\end{lemm}

\begin{proof}
Due to Eq.~\eqref{fibers} and the definition of $\mathcal{X}$, we have $\Delta^\prime=\alpha r_g$ or $\Delta^\prime=\alpha s_g$ for some $g\in C$. In the former case, $\Delta^\prime$ is a singleton by Lemma~\ref{scheme}(1). So $n_t=1$ for every $t\in T_{\Delta,\Delta^\prime}$ and we are done.

Suppose that $\Delta^\prime=\alpha s_g$. Lemma~\ref{distinct} implies that $|\Delta\cap \Gamma|=|\Delta^\prime\cap \Gamma|=1$ for every $\Gamma\in \Omega/r_C$ such that $\alpha \notin \Gamma$. So $r_C\cap (\Delta\times \Delta^\prime)\neq \varnothing$ and hence $r_h\cap (\Delta\times \Delta^\prime)\neq \varnothing$ for some $h\in C$. Denote the latter relation by $t$. Lemma~\ref{onepoint}(1) yields that $t$ is a relation of $\mathcal{Y}$. Due to Lemma~\ref{scheme}(1), we have $n_{r_h}=1$. Together with the obvious inequality $n_t\leq n_{r_h}$, this implies that $n_t=1$. The latter means that $t$ is a basis relation of $\mathcal{Y}$. Thus, the lemma holds for $t$ as required.
\end{proof}

\begin{lemm}\label{deltareg}
With the above notation, $\mathcal{Y}_{\Delta}$ is isomorphic to the cyclotomic scheme $\cyc(q,n)$.
\end{lemm}

\begin{proof}
Let $g\in C$, $r=r_g$, and $s=s_g$. Then using Eq.~\eqref{deltaset}, one can deduce that $r_{\Delta}\neq \varnothing$ if and only if $g=e$ and $r_{\Delta}=1_{\Delta}$, and
$$s_{\Delta}=\{(K(x,1),K(y,1)):~x,y\in \mathbb{F},~x-y\in g\}.$$
Therefore the pair $\mathcal{Y}^\prime=(\Delta,S_{\Delta})$ is a scheme isomorphic to $\cyc(q,n)$. Lemma~\ref{onepoint}(1) implies that $\mathcal{Y}^\prime\leq \mathcal{Y}_{\Delta}$.

To prove the lemma, it remains to verify that $\mathcal{Y}^\prime\geq \mathcal{Y}_{\Delta}$. Let 
$$\mathcal{M}=\left\{
\left (\begin{smallmatrix}
k & b \\
0 & 1\\
\end{smallmatrix}\right ):~k\in K,~b\in \mathbb{F}\right\}.$$
Observe that $\mathcal{M}$ is a group with respect to the matrix multiplication, this group induces the group $\mathcal{M}^\Omega$ acting on $\Omega$, every element of $\mathcal{M}^\Omega$ fixes $\alpha$, and $\mathcal{M}^\Omega\leq \GL(2,q)_K^\Omega$. The latter inequality together with Lemma~\ref{combiso}(2) implies that $\mathcal{M}^\Omega\leq \aut(\mathcal{X})$. So $\mathcal{M}^\Omega\leq \aut(\mathcal{X})_{\alpha}$. Lemma~\ref{onepoint}(2) yields that $\aut(\mathcal{X})_{\alpha}=\aut(\mathcal{Y})$. Therefore
$$\mathcal{M}^\Omega\leq \aut(\mathcal{Y}).$$
Due to the latter inequality, we have 
$$\inv(\mathcal{M}^\Omega)\geq \inv(\aut(\mathcal{Y}))\geq \mathcal{Y}$$
and hence
\begin{equation}\label{in2}
\inv(\mathcal{M}^\Omega)_{\Delta}\geq \mathcal{Y}_{\Delta}.
\end{equation}
From Lemma~\ref{cyclschurian} and the definition of $\mathcal{M}$ it follows that $\inv(\mathcal{M}^\Omega)_{\Delta}\cong \cyc(q,n)$ and consequently $\inv(\mathcal{M}^\Omega)_{\Delta}=\mathcal{Y}^\prime$. Together with Eq.~\eqref{in2}, this implies that $\mathcal{Y}^\prime\geq \mathcal{Y}_{\Delta}$ as desired. 
\end{proof}

Let $\beta,\gamma\in \Delta$ with $\beta\neq \gamma$. Put $\mathcal{Z}=\mathcal{X}_{\alpha,\beta}=\mathcal{Y}_{\beta}$ if $q\notin \mathcal{Q}_0$ and $\mathcal{Z}=\mathcal{X}_{\alpha,\beta,\gamma}=\mathcal{Y}_{\beta,\gamma}$ otherwise. Observe that $\Delta\in F(\mathcal{Z})^\cup$ because $\mathcal{Z}\geq \mathcal{Y}$.

\begin{lemm}\label{seprestr}
With the above notation, $\mathcal{Z}_{\Delta}$ is separable. 
\end{lemm}

\begin{proof}
By the definition of $\mathcal{Z}$, we have 
$$\mathcal{Z}_{\Delta}\geq (\mathcal{Y}_{\Delta})_{\beta}$$
if $q\notin \mathcal{Q}_0$ and 
$$\mathcal{Z}_{\Delta}\geq (\mathcal{Y}_{\Delta})_{\beta,\gamma}$$
otherwise. Since $\mathcal{Y}_{\Delta}$ is isomorphic to $\cyc(q,n)$ (Lemma~\ref{deltareg}), the above inequalities and Lemma~\ref{cyclsepar} imply that $\mathcal{Z}_{\Delta}$ is separable. 
\end{proof}

\begin{lemm}\label{sepz}
With the above notation, $\mathcal{Z}$ is separable. 
\end{lemm}

\begin{proof}
Let $\Delta^\prime\in F(\mathcal{Z})$ be such that $\Delta^\prime\not\subseteq \Delta$. There exists $\Delta^{\prime\prime}\in F(\mathcal{Y})$ such that $\Delta^\prime\subseteq \Delta^{\prime\prime}$. Clearly, $\Delta^{\prime\prime}\neq \Delta$. So there exists $t\in T_{\Delta,\Delta^{\prime\prime}}$ such that $n_t=1$ by Lemma~\ref{fibere}. Note that $t_{\Delta,\Delta^\prime}=t\cap (\Delta\times\Delta^\prime)$ is a relation of $\mathcal{Z}$. As $n_t=1$, we have $n_{t_{\Delta,\Delta^\prime}}=1$ and hence there is a basis relation of valency~$1$ inside $\Delta\times \Delta^\prime$. This yields that $\Delta$ satisfies the condition of Lemma~\ref{deletefiber}. Therefore $\mathcal{Z}$ is separable whenever so is $\mathcal{Z}_{\Delta}$. The latter follows from Lemma~\ref{seprestr} and we are done.
\end{proof}

Due to Lemma~\ref{sepz}, we have $s(\mathcal{Z})=1$. Lemma~\ref{s1} implies that
$$s(\mathcal{X})\leq s(\mathcal{Z})+2=3$$
if $q\notin \mathcal{Q}_0$ and
$$s(\mathcal{X})\leq s(\mathcal{Z})+3=4$$
otherwise as desired.

If $n=q-1$, then $\mathcal{Y}_{\Delta}$ is regular. So it is separable by Lemma~\ref{regularsep}. Together with Lemma~\ref{deletefiber} and Lemma~\ref{fibere}, this implies that $\mathcal{Y}$ is separable or, equivalently, $s(\mathcal{Y})=1$. Therefore $s(\mathcal{X})\leq 2$ by Lemma~\ref{s1} which proves Statement~$(2)$ of Theorem~\ref{main}.

It remains to prove Statement~$(3)$ of Theorem~\ref{main}, i.e., that $s(\mathcal{X})>1$ whenever $n\geq 3$ and $r=\Char(\mathbb{F}_q)$ is not a primitive root modulo $n$. To do this, it suffices to show that $\mathcal{X}$ is nonseparable in this case. The latter condition implies that there exist generators $g,g^\prime$ of $C\cong C_n$ such that $g^\prime\neq g^{r^i}$ for every $i\in \mathbb{Z}$ and hence
\begin{equation}\label{fieldiso} 
g^\prime\neq g^\sigma
\end{equation}
for every $\sigma\in \Sigma$. 

Since $g$ and $g^\prime$ are generators of $C$, there exists $\sigma_0\in \aut(C)$ such that $g^\prime=g^{\sigma_0}$. Let $\varphi=\varphi_{\sigma_0,e}$. Due to Eq.~\eqref{defnf}, we have 
\begin{equation}\label{aux}
r_{g^\prime}=r_{g^{\sigma_0}}=r_g^\varphi.
\end{equation} 
From Lemma~\ref{algiso} it follows that $\varphi\in \aut_{\alg}(\mathcal{X})$. Assume that $\varphi$ is induced by some $f\in \iso(\mathcal{X})$. Lemma~\ref{combiso}(3) implies that $f=(T\sigma)^\Omega$ for some $T\in \GL(2,q)$ and $\sigma\in \Sigma$. Therefore 
$$r_{g^\prime}=r_g^\varphi=r_g^f=r_{g^\sigma},$$ 
where the first equality holds by Eq.~\eqref{aux}, the second one holds because $\varphi$ is induced by $f$, whereas the third one holds by Lemma~\ref{combiso}(1). Thus, $g^\prime=g^\sigma$, a contradiction to Eq.~\eqref{fieldiso}.

\end{document}